\documentclass[dvips,12pt]{iopart}
\usepackage{iopams} 
\usepackage{mathrsfs} 
\usepackage{amssymb} 
\usepackage{amsthm} 
\usepackage{graphicx}
\eqnobysec
\newtheorem{proposition}{Proposition}[section]

\newtheorem{lemma}{Lemma}[section]
\newtheorem{theorem}{Theorem}
\newtheorem{cor}{Corollary}[section]
\newtheorem{example}{Example}

\theoremstyle{plain}
\newtheorem*{kats}{Kats' Theorem}
\newtheorem{rem}{Remark}

\begin{document}

\title[Krein's inverse problem and continued fractions]{Continued fraction solution
of Krein's inverse problem}

\author{Yves Tourigny}
\address{School of Mathematics\\
        University of Bristol\\
        Bristol BS8 1TW, United Kingdom}
\ead{y.tourigny@bristol.ac.uk}

\begin{abstract}
The spectral data of a vibrating string are encoded in its so-called characteristic function.
We consider the problem of recovering the distribution of mass along the string from its characteristic function.
It is well-known that Stieltjes' continued
fraction leads to the solution of this inverse problem in the particular case where the distribution of mass
is purely discrete.
We show how to adapt Stieltjes' method to solve the inverse problem for a related class of strings. An application to the excursion theory of
diffusion processes is presented.
\end{abstract}

\submitto{\IP}

\section{Introduction}
Consider an inextensible string of unit tension placed along the $x$ axis with its left endpoint at $0$, and let $M(x)$ be the total mass of the string
segment from $0$ to $x$. The small vertical oscillations $y = y(x)$ of the string obey the equation
\begin{equation}
y'' + z M' y = 0
\label{stringEquation}
\end{equation}
where $z$ is the square of the oscillation frequency. The natural frequencies
of the string depend on how it is tied at the ends; suppose
\begin{equation}
y'(0-) = y(\ell) + (L-\ell) \,y'(\ell+) = 0
\label{boundaryCondition}
\end{equation}
where $\ell$ is the supremum of the points of increase of $M$, and $\ell \le L \le \infty$ is some ``tying constant''. 

Now, let $z \in {\mathbb R}_-$, and denote by
$\varphi(\cdot,z)$ and $\psi(\cdot,z)$ the two particular solutions of Equation (\ref{stringEquation}) satisfying
$$
\varphi(0,z) = \psi'(0-,z) = 1 \quad \mbox{and} \quad \varphi'(0-,z) = \psi(0,z) = 0\,.
$$
The {\em characteristic function} of the string is then defined by
\begin{equation}
W(z) = \lim_{x \rightarrow L} \frac{\psi(x,z)}{\varphi(x,z)}
\label{characteristicFunction}
\end{equation}
or, equivalently, 
\begin{equation}
W(z) = -\frac{w(0,z)}{w'(0-,z)}
\label{weylFunction}
\end{equation}
where $w(\cdot,z)$ is, up to an unimportant factor, the unique non-negative solution of the string equation that is decreasing and satisfies $w(L,z) = 0$.
The definition of $W$ may be extended to ${\mathbb C} \backslash {\mathbb R}_+$ by analytic continuation; $W$ is the analog of the Weyl--Titchmarsh function in the Sturm--Liouville theory.

The present paper is devoted to the inverse spectral problem for this vibrating string equation, namely the problem of recovering $M$ from $W$.
We call this {\em Krein's inverse problem}. We describe and illustrate a novel algorithm that is
applicable to a class of characteristic functions associated with a certain continued fraction. The remainder of this introduction provides a summary of
the approach and discusses its relationship with other works. 

\subsection{Krein's inverse problem}
\label{kreinSubsection}
In a series of papers published in the Soviet Union during the 1950's, M. G. Krein made a detailed study of the existence of spectral
expansions associated with the vibrating string equation. For some useful accounts of this work in the English language, see
\cite{DM,KK,KW}; in particular, \cite{KK} contains a rough outline of the
historical development of Krein's ideas and their overlap with the work of Feller aimed at
a unified analytical treatment of some discrete and continuous stochastic processes.

Denote by
${\mathbb M}$ the set of functions $M \,: [0,\infty] \rightarrow [0,\infty]$ that are non-decreasing, right-continuous and infinite
at infinity. We use the convention $M(0-)=0$ so that $d M (x)$ is a well-defined Stieltjes measure.
The tying constant $L$ is absorbed in the definition of $M$ by setting
$$
L := \sup \left \{ x \in {\mathbb R}_+ : \,M(x) < \infty \right \}\,.
$$
For such $M$, Krein gave a precise meaning to Equation (\ref{stringEquation}) and its particular solutions
$\varphi(\cdot,z)$, $\psi(\cdot,z)$ and $w(\cdot,z)$. With a slight abuse of terminology, we shall identify $M$ with the string.
Krein showed that the characteristic function defined by equation (\ref{characteristicFunction}) is necessarily
of the form
\begin{equation}
W(z) = c + \int_{[0,\,\infty)} \frac{d \sigma (\lambda)}{\lambda-z} = c + \frac{s}{-z} +  \int_0^\infty \frac{d \sigma (\lambda)}{\lambda-z}
\label{characteristicRepresentation}
\end{equation}
where $c \ge 0$, $\sigma$ is a right-continuous non-decreasing function with support in $[0,\infty)$ such that
\begin{equation}
\int_{[0,\infty)} \frac{d \sigma (\lambda)}{1+\lambda} < \infty
\label{kreinCondition}
\end{equation}
and $s = \sigma(0)$.
The paramater $c$ is in fact the infimum of the points of increase of $M$, and the function $\sigma$ in this representation is called the
{\em principal spectral function} of the string.
We shall denote by ${\mathbb W}$ the set of {\em all} functions expressible in this form. Hence Krein showed that $W \in {\mathbb W}$ for every
$M \in {\mathbb M}$. He also conjectured the converse: {\em every} element of ${\mathbb W}$ is the characteristic function of a unique string in ${\mathbb M}$.
To the best of our knowledge, the first proof of this conjecture was given in \cite{DM}, Chapter 6; it uses the
theory of Hilbert spaces of entire functions, and a deep uniqueness theorem due to de Branges. 

\subsection{The Sturm--Liouville problem in impedance form}
\label{sturmLiouvilleSubsection}
Although there is no known systematic procedure for solving
Krein's inverse problem in its full generality, effective algorithms do exist if the mass distribution belongs to certain special classes. 
For instance,
if $M$ is smooth, then the change of variables
$$
t = T(x) := \int_0^x \sqrt{M'(s)}\, d s\,, \;\; u(t) = y \left ( X(t) \right )
$$ 
where $X$ is the inverse of $T$, transforms the string equation into 
\begin{equation}
\frac{d}{d t} \left ( \varrho \, \frac{du}{dt} \right ) + z \,\varrho \, u = 0
\label{impedanceForm}
\end{equation}
where
\begin{equation}
\varrho  = \sqrt{M' \circ X}
\label{impedance}
\end{equation}
is the so-called impedance.
The inverse problem for this  equation may be tackled by  ``integral operator'' techniques going back to the works of Gelfand, Krein, Levitan and Marchenko; see \cite{Ma} for an informal account of these techniques and their historical
background, and \cite{AHM} for more recent developments.
General methods inspired by the Gelfand--Levitan approach aimed at solving the inverse problem for the string equation have been considered by
Dym \& Kravitsky \cite{DK1,DK2}, Winkler \cite{Wi} and Boumenir \cite{Bo}.

\subsection{The Stieltjes moment problem and the class ${\mathbb M}^+$}
\label{stieltjesSubsection}
Another class for which the inverse problem is tractable arises in connection with
{\em Stieltjes'  moment problem} \cite{St}:
given a sequence of numbers $c_0,\,c_1,\,\ldots$, find a measure $d \sigma$ on ${\mathbb R}_+$ such that
\begin{equation}
c_n = \int_0^\infty \lambda^n \,d \sigma(\lambda) \quad \mbox{for $n=0,\,1,\,\ldots$}\,.
\label{stieltjesMomentProblem}
\end{equation}
Briefly, Stieltjes' approach was to construct from the $c_j$ a certain continued fraction, namely
\begin{equation}
\frac{1}{\displaystyle - s_0 z + \frac{1}{\displaystyle s_1 + \frac{1}{\displaystyle -s_2 z + \cdots}}}\,.
\label{stieltjesContinuedFraction}
\end{equation}
Stieltjes gave an algorithm that permits the calculation of the $s_n$ from the $c_n$. The coefficients $s_n$ thus obtained form a finite
or infinite sequence of strictly positive numbers, depending on whether the measure $\sigma$ has finitely or infinitely many points of growth.
Stieltjes distinguished the {\em determinate} case, in which the moment problem has a unique solution,
and its complement, the {\em indeterminate} case. 
The determinate case arises if and only if the continued fraction either terminates (i.e. the sequence of the $s_n$ is finite) or else
is infinite but convergent. The continued fraction is then a well-defined function of $z$ which may be expressed as the
Stieltjes transform of the sought measure $\sigma$. Stieltjes also found a beautifully simple criterion for the convergence of the
continued fraction when the sequence of the $s_n$ is infinite; he showed that
convergence occurs if and only if
\begin{equation}
\sum_{n=0}^{\infty} s_n = \infty\,.
\label{stieltjesSeries}
\end{equation}

This work on the moment problem provided the inspiration for Krein's own theory of strings; see Supplement II in \cite{GK}.  With remarkable insight, Krein was able to relate
Stieltjes' continued fraction to the characteristic function of a {\em discrete} string, i.e. a string made up of a sequence
of point masses $m_j$ at the positions
$x_j$. To see how, set
$$
W(x,z) := -\frac{w(x,z)}{w'(x,z)}\,.
$$
The string equation then yields a Riccati equation for $W(\cdot,z)$:
$$
W'(x,z) = -1 - z M'(x) W^2(x,z)\,.
$$
For a string consisting of point masses, $M'$ is a linear combination of Dirac deltas; at a point $x_j$ with concentrated mass $m_j$, this Riccati equation
should be interpreted as
$$
W(x_j-,z) = W(x_j+,z) + z \,m_j W(x_j-,z) W(x_j+,z)\,.
$$
By
using Equation (\ref{weylFunction}), we find
\begin{equation*}
W(z) = W(0-,z) = x_0 + W(x_0-,z) = x_0 + \frac{1}{\displaystyle - m_0 z + \frac{1}{\displaystyle W(x_0+,z)}}
\end{equation*}
and, by iterating,
\begin{equation}
W(z)  =  x_0 + \frac{1}{\displaystyle - m_0 z + \frac{1}{\displaystyle x_1-x_0 + \cdots + \frac{1}{\displaystyle -m_n z + \frac{1}{\displaystyle  W(x_n+,z)}}}}\,.
\label{kreinExpansion}
\end{equation}
If the principal spectral function has finitely many points of increase, then the string consists of finitely many, say $n$, point masses with finite mass (and a further point of infinite mass at $L$). Then $W(x_n+,z) = L-\ell$, the expansion terminates, and we have the finite case
of Stieltjes' continued fraction (\ref{stieltjesContinuedFraction}). If the principal spectral function has infinitely many points of
increase, then there are infinitely many point masses; the continued fraction does not terminate
but its coefficients obviously still determine $M(x)$ for every $x$ such that
$$
x < \ell_0 := \lim_{n \rightarrow \infty} x_n \,.
$$
In particular, if
$$
\ell_0 + \sum_{j=0}^\infty m_j = \infty
$$
then the continued fraction converges and the string is {\em completely} determined by its coefficients since either 
$\ell_0 = \infty$ or else $L = \ell_0$.

Let us now elaborate the significance of the foregoing remarks for Krein's inverse problem. Define
\begin{equation}
{\mathbb W}^+ := \left \{ W \in {\mathbb W} :\, \int_0^\infty \lambda^n \, d \sigma (\lambda) < \infty \quad \mbox{for $n =0,\,1,\,2,\,\ldots$} \right \}
\label{Wplus}
\end{equation}
and denote by ${\mathbb M}^+$ the corresponding set of strings. Every $W \in {\mathbb W}^+$ may be expanded
in a continued fraction of the form (\ref{kreinExpansion}), and the coefficients in the expansion may be computed from the moments
of the principal spectral function $\sigma$. It will serve
our purpose to abuse Stieltjes' terminology somewhat by saying that $W$ itself belongs to the {\em determinate subclass} of ${\mathbb W}^+$
if the continued fraction is either finite or convergent, and to the {\em indeterminate subclass} otherwise.
It follows, then, that Stieltjes' continued fraction provides an effective solution method of the inverse problem for strings in ${\mathbb M}^+$
whose characteristic functions are in the determinate subclass. 

\begin{rem}
Not every member of ${\mathbb M}^+$ is a discrete string; all that can be said is that such strings must begin with a sequence
of point masses; see \cite{DM}, \S 5.9. If $W \in {\mathbb W}^+$ is in the determinate subclass, then $M$ is a discrete string. The converse is of course not true, and for a discrete
string whose characteristic function is in the indeterminate subclass, one cannot recover the tying constant from the continued fraction coefficients
alone; see \cite{KK}, \S 13.
\label{discreteRemark}
\end{rem}

\subsection{Main results and outline of the paper}
\label{resultsSubsection}
The connection between discrete strings and continued fractions can be exploited to provide a reconstruction algorithm of greater applicability. Beals \etal use this idea in the context of integrable systems \cite{BSS1,BSS2}; 
Borcea \etal use it as the basis of their numerical treatment of the inverse Sturm--Liouville problem in impedance form \cite{BDK}. 
The practical issue that arises is how
to find approximations of the spectral data that satisfy Stieltjes' moment condition and have explicitly computable moments. 
Both Beals \etal and Borcea \etal resolve this issue by restricting their attention
to the case of a {\em discrete spectrum}, where the characteristic function
has a Mittag--Leffler expansion which may be truncated to furnish the required approximations.

We present an alternative truncation strategy applicable to a class that includes some characteristic functions with an absolutely continuous spectrum.
This class is a proper subset of a class denoted ${\mathbb M}^-$, based on a continued fraction expansion
that arises naturally in connection with Krein's inverse problem. The expansion rests on the following
two useful ``rules'' relating the characteristic function to its string.
The first of these rules concerns the characteristic function of the right-continuous inverse $M^\ast$ of the string $M$:
\begin{proposition}
$$
W^\ast(z) := \frac{1}{-z W(z)}\,.
$$
\label{dualProposition}
\end{proposition}
The string $M^\ast$ will be called the {\em dual} of $M$.
The second rule concerns the string obtained from $M$ by removing $0$ from the spectrum of $\sigma$:
\begin{proposition}
Let $M$ be the string with characteristic function
$$
W (z) = c + \int_{[0,\infty)} \frac{\mbox{\em d} \sigma (\lambda)}{\lambda-z}\,.
$$
Then
$$
\widehat{W} (z) := c + \int_0^\infty \frac{\mbox{\em d} \sigma (\lambda)}{\lambda-z} = W(z) + \frac{s}{z}
$$
is the characteristic function of the string $\widehat{M}$ defined by
$$
\widehat{M}(x) := \frac{M(t)}{1-\frac{M(t)}{M(\infty-)}}\,, \;\;\mbox{where}\; x = \int_0^t \left [ 1- \frac{M(\tau)}{M(\infty-)} \right ]^2\,\mbox{\em d} \tau\,.
$$
\label{zeroMassProposition}
\end{proposition}
Two other useful formulae connecting $W$ and $M$ are:
\begin{equation}
M(\infty-) = \frac{1}{s}
\label{atom}
\end{equation}
and, in the case $\ell + M(\ell-) < \infty$,
\begin{equation}
L = c+\int_{[0,\infty)} \frac{d \sigma (\lambda)}{\lambda}\,.
\label{tyingConstant}
\end{equation}
For proofs of these results,
see \cite{DM}, \S 6.8 and \S 6.9. Using these rules and setting
$W_0(z) := W(z)$, we may write
\begin{equation*}
W_0 (z) = \frac{s_0}{-z} + \widehat{W}_0(z) = \frac{s_0}{-z} + \frac{1}{-z W_1(z)}
\end{equation*}
where $W_1$ is the characteristic function of the dual string, say $M_1$, of $\widehat{M}_0$. By the same argument,
\begin{equation*}
W(z) = \frac{s_0}{-z} + \frac{1}{\displaystyle s_1 + \frac{1}{\displaystyle W_2 (z)}}
\end{equation*}
where $W_2$ is the characteristic function of the dual string, say $M_2$, of $\widehat{M}_1$. Iterating, we obtain 
\begin{equation}
W(z) = \frac{s_0}{-z} + \frac{1}{\displaystyle s_1+\cdots+\frac{1}{\displaystyle \frac{s_{2n-2}}{-z}+ \frac{1}{\displaystyle -z W_{2n-1}(z)}}}\,.
\label{katsContinuedFraction}
\end{equation}
This continued fraction, albeit in a very particular context and in a somewhat disguised form, was introduced and studied in our recent work on the excursions of diffusion processes \cite{CT}.
It is also implicit in some of the calculations carried out by Donati--Martin and Yor in their generalisation 
of L\'{e}vy's formula for the area enclosed by a planar Brownian motion  \cite{DY}.

In \S \ref{katsSection}, we study the relationship between this continued fraction and that of Stieltjes. This relationship may be described in terms of the
map
\begin{equation}
W \mapsto W^{-}(z) := \frac{1}{-z} W^\ast \left ( \frac{1}{z} \right )\,.
\label{minusMap}
\end{equation}
Then
\begin{equation*}
W(z) = \frac{s_0}{-z} + \frac{1}{\displaystyle s_1+\frac{1}{\displaystyle \frac{s_2}{-z}+ \cdots}} \iff
W^-(z)  =  \frac{1}{\displaystyle - s_0 z + \frac{1}{\displaystyle s_1 + \frac{1}{\displaystyle -s_2 z + \cdots}}}\,.
\end{equation*}
Hence the class
\begin{equation}
{\mathbb W}^- := \left \{ W^- :\, W \in {\mathbb W}^+ \right \}
\label{Wminus}
\end{equation}
is the set of characteristic functions that plays, for the continued fraction (\ref{katsContinuedFraction}), the part played by the class
${\mathbb W}^+$ for  the continued fraction of Stieltjes. We show by elementary means that the corresponding set of strings, denoted ${\mathbb M}^-$, consists
of those strings whose {\em dual} principal spectral functions have {\em negative} moments.
Then, by a straightforward application of Kats' results on such spectral functions \cite{Ka}, we derive a simple criterion, expressed in terms of the mass distribution $M(x)$, for $M$ to be in the class ${\mathbb M}^-$.
The upshot is that this class is large enough to be of interest; in particular, it contains some strings
with an absolutely continuous principal spectral function.

Just as in the case of the Stieltjes continued fraction, ${\mathbb W}^-$ may be partitioned into two 
subclasses: determinate and indeterminate. 
The determinate subclass consists of all those characteristic functions for which the continued fraction expansion (\ref{katsContinuedFraction}) is either
finite or else convergent--- the latter case occuring if and only if Equation (\ref{stieltjesSeries}) holds. 
\S \ref{inverseSection} discusses the finite case. We show there that
the string $M$ corresponding to such a characteristic function is discrete with finitely many point masses, and we devise an algorithm to compute
it explicitly, given the continued fraction coefficients. The infinite (convergent) case may then be tackled by considering the truncations of the
continued fraction after $n$ terms; since Krein's correspondence between strings and characteristic functions is a homeomorphism,
the discrete string so obtained yields in the limit $n \rightarrow \infty$ the solution of the inverse problem.

In \S \ref{diffusionSection}, we give some examples that arise  in the study of the excursions of one-dimensional diffusion processes. Knight \cite{Kn} 
and Kotani \& Watanabe \cite{KW} pointed out simultaneously the relevance of Krein's theory to this topic:
every string $M$ defines a {\em generalised diffusion process}, and the excursions of the process from its starting point may be described by means of the principal sprectral function of the dual string. The probabilistic version of Krein's inverse problem is to find the diffusion, given the distribution of the excursion lengths. Donati--Martin and Yor solved this problem explicitly in a number of interesting cases \cite{DY1,DY2}, and we use
their findings to illustrate the effectiveness of our algorithm.

The paper ends with a few concluding remarks in \S \ref{conclusionSection}.

\section{The class ${\mathbb M}^-$}
\label{katsSection}

\begin{lemma}
For every $W \in {\mathbb W}$, $W^- \in {\mathbb W}$. Furthermore, the map
$$
W \mapsto W^{-}
$$
from ${\mathbb W}$ to itself is an involution.
\label{involutionLemma}
\end{lemma}

\begin{proof}
Let $W \in {\mathbb W}$. Since $W^{\ast} \in {\mathbb W}$, we can write
$$
W^\ast (z) = c^\ast + \frac{s^{\ast}}{-z} + \int_0^\infty \frac{d \sigma^{\ast}(\lambda)}{\lambda-z}
$$
where $c^\ast \ge 0$, $s^{\ast} \ge 0$ and
$$
\int_0^\infty \frac{d \sigma^{\ast}(\lambda)}{\lambda+1} <\infty\,.
$$
A simple calculation shows that
$$
W^- (z) = s^\ast  + \frac{c^\ast}{-z} + \int_0^\infty \frac{d \sigma^- (\lambda)}{\lambda-z}
$$
where
\begin{equation}
d \sigma^- (\lambda) = \frac{d \sigma^\ast (1/\lambda)}{\lambda}\,.
\label{dualMeasure}
\end{equation}
It is then readily seen that
$$
\int_0^{\infty} \frac{d \sigma^- (\lambda)}{\lambda+1} = \int_0^\infty \frac{d \sigma^{\ast}(\lambda)}{\lambda+1}
$$
and so $W^- \in {\mathbb W}$. The fact that
$$
\left ( W^- \right )^- = W
$$
then follows easily.
\end{proof}

\begin{theorem}
$$
{\mathbb W}^- = \left \{ W \in {\mathbb W} :\, \int_0^\infty \lambda^{-n} \,d \sigma^\ast (\lambda) < \infty \quad \mbox{\em for $n =1,\,2,\,\ldots$} \right \}  \,.
$$
\label{momentTheorem}
\end{theorem}

\begin{proof}
The fact that $W \mapsto W^-$ is an involution implies
$$
{\mathbb W}^- := \left \{ W^- :\, W \in {\mathbb W}^+ \right \} = \left \{ W \in {\mathbb W} :\, W^- \in {\mathbb W}^+ \right \}\,.  
$$
Note that,
by definition of ${\mathbb W}^+$, $W^- \in {\mathbb W}^+$ if and only if
$$ 
\int_0^\infty \lambda^n \, d \sigma^- (\lambda) < \infty \quad \mbox{for $n=0,\, 1,\, \ldots$}
$$
The theorem is then a consequence of Equation (\ref{dualMeasure}).
\end{proof}
\begin{cor}
$$
M \in {\mathbb M}^- \iff M^\ast \in {\mathbb M}^-\,.
$$
\label{dualCorollary}
\end{cor}
\begin{proof}
The dual of a discrete string is a discrete string. Hence
$$
M \in {\mathbb M}^+ \iff M^\ast \in {\mathbb M}^+\,.
$$
Furthermore, the maps 
$$
M \mapsto M^- \quad \mbox{and} \quad M \mapsto M^\ast
$$
commute.
\end{proof}

Next, we seek a simple criterion--- in terms of the mass distribution--- for $M$ to belong to the class ${\mathbb M}^-$. We shall need the following
key result, due to Kats \cite{Ka}:

\begin{kats}
For the principal spectral function $\sigma$ of a string $M \in {\mathbb M}$ to satisfy
$$
\int_{[0,\infty )} \lambda^{-n} \,d \sigma (\lambda ) < \infty \quad \mbox{for every $n = 1,\,2,\,\ldots$}
$$
it is necessary and sufficient that $L<\infty$ and that 
$$
\forall \; \varepsilon > 0, \;M(x) \left [L-x\right ]^{1 + \varepsilon} = o(1) \quad \mbox{as $x \rightarrow L-$}\,.
$$
\end{kats}

\begin{theorem}
$M \in {\mathbb M}^-$ if and only if either
\begin{enumerate}
\item $M(\infty-) < \infty$ and $\forall \; \varepsilon > 0, \;\lim_{x \rightarrow \infty} x \left [ M(\infty-) - M(x) \right ]^{1+\varepsilon} = 0$ \\
or 
\item $L < \infty$ and $\forall \; \varepsilon > 0, \;\lim_{x \rightarrow L-} M(x) \left  [ L-x \right ]^{1+\varepsilon} = 0$.
\end{enumerate}
\label{stringTheorem}
\end{theorem}

\begin{proof}
We note first that, for $M \in {\mathbb M}^-$, the implication 
$$
s=0 \Rightarrow s^\ast >0
$$ holds because
$W^-$ has an expansion of the form (\ref{kreinExpansion}). Furthermore $M(\infty-) = \frac{1}{s}$.  Hence 
we have the trivial identity
$$
M \in {\mathbb M}^{-} \iff \left [ M \in {\mathbb M}^{-} \;\mbox{and}\; M(\infty-) < \infty \right ] \;\mbox{or}\; \left [ M \in {\mathbb M}^{-} \;\mbox{and}\; M(\infty-) = \infty  \right ] 
$$
$$
\iff \left [ M \in {\mathbb M}^{-} \;\mbox{and}\; M(\infty-) < \infty \right ] \;\mbox{or}\; \left [ M^\ast \in {\mathbb M}^{-} \;\mbox{and}\; M^\ast(\infty-) < \infty  \right ]
\,.
$$
Let us show first that 
\begin{equation}
\left [ M^\ast \in {\mathbb M}^{-} \;\mbox{and}\; M^\ast(\infty-) < \infty \right ] \;\iff\; \mbox{Condition (ii) holds}\,.
\label{firstHalf}
\end{equation}
Suppose that  $M^{\ast} \in {\mathbb M}^{-}$ and  $M^{\ast}(\infty-) < \infty$. Then
$$
\ell = L = M^{\ast}(\infty-) < \infty
$$
and so
$$
s = 0\,.
$$
It follows from Corollary \ref{dualCorollary} that
$$
\int_{[0,\,{\infty})} \lambda^{-n} \,d \sigma(\lambda) < \infty \quad \mbox{for $n=1,\,2,\,\ldots$}\,. 
$$
Kats' Theorem then implies that
$$
\forall \; \varepsilon > 0\,, \quad M(x) \left [ L-x\right ]^{1 + \varepsilon} = o(1) \quad \mbox{as $x \rightarrow L-$}\,.
$$
This is Condition (ii). Conversely, if Condition (ii) holds,
we arrive by a similar argument at the conclusion that $M^{\ast} \in {\mathbb M}^{-}$ and  $M^\ast(\infty-) < \infty$. This proves (\ref{firstHalf}). 
To complete the proof of the theorem, it suffices to replace $M$ by $M^\ast$ in (\ref{firstHalf}).
\end{proof}

\begin{rem}
The problem of characterising
the {\em determinate subclass} in terms of the mass distribution appears to be much more difficult.
\label{criterionRemark}
\end{rem}

\begin{example}
Let $u(\cdot,z)$ be the particular
solution of Equation (\ref{impedanceForm}) in the interval $[0,1]$ such that
$$
u'(0,z)=-1 \;\;\mbox{and} \;\; u(1,z) = 0\,.
$$
The problem considered by Borcea \etal is to recover $\varrho$, normalised by $\varrho(0)=1$, from $u(0,z)$ \cite{BDK}. Now, 
suppose that
$$
L := \int_0^1 \frac{d \tau}{\varrho (\tau)} < \infty \;\;\mbox{and}\;\; \int_0^1 \varrho (\tau) \,d \tau < \infty\,.
$$
Then $u(0,z)$ is the characteristic function of a string $M$ such that
$$
\ell = L < \infty \;\;\mbox{and} \;\; M(L-) < \infty\,.
$$
It is an easy consequence of Theorem \ref{stringTheorem} that this string belongs to ${\mathbb M}^-$.
Knowing $M$, $\varrho$ can in principle be obtained via Equation (\ref{impedance}).

Given the continued fraction coefficients $s_n$, one can ascertain whether $W$ is in the determinate class
by using Stieltjes criterion (\ref{stieltjesSeries}). To give a more concrete example, consider the case
$$
\varrho \equiv 1\,.
$$
Then
$$
W(z) = \frac{\mbox{\rm th} \left ( \sqrt{-z} \right )}{\sqrt{-z}} = 
\frac{1}{\displaystyle 1 + \frac{1}{\displaystyle \frac{3}{-z} + \frac{1}{\displaystyle 5 + \cdots}}}\,.
$$
The series
$$
\sum_{n=0}^\infty s_n = \sum_{j=0}^\infty \left ( 2 j+1 \right )
$$
diverges, and so $W$ is in the determinate class.

\label{impedanceExample}
\end{example}

\section{The inversion algorithm}
\label{inverseSection}

We proceed to compute the string corresponding to the finite case of the continued fraction expansion on the right-hand
side of Equation (\ref{katsContinuedFraction}). In what follows, it will always be assumed that
$$
s_0 \ge 0 \;\;\mbox{and} \;\; s_j > 0 \;\mbox{if $j >0$}\,.
$$
We shall use the compact notation
$$
W[s_j,\,\ldots,\,s_{j+2 k}] := \frac{s_j}{-z} + \frac{1}{\displaystyle s_{j+1}+ \frac{1}{\displaystyle \frac{s_{j+2}}{-z}+ \cdots + \frac{1}{\displaystyle s_{j+2k-1}+ \frac{1}{\displaystyle \frac{s_{j+2k}}{-z}}}}}
$$
and
$$
W[s_j,\,\ldots,\,s_{j+2k+1}] := \frac{s_j}{-z} + \frac{1}{\displaystyle s_{j+1}+ \frac{1}{\displaystyle \frac{s_{j+2}}{-z}+ \cdots + \frac{1}{\displaystyle s_{j+2k-1}+ \frac{1}{\displaystyle \frac{s_{j+2k}}{-z}+\frac{1}{\displaystyle s_{j+2k+1}}}}}}
$$
and an analogous notation for the corresponding strings. These characteristic functions are in the determinate subclass of ${\mathbb W}^-$
because their images under the map $W \mapsto W^-$ are in the determinate subclass of ${\mathbb W}^+$.

\begin{theorem}
For every $k > 0$,
$$
M[s_k] = \frac{1}{s_k}
$$
and, for every $0 \le j < k$,
\begin{equation}
M [ s_j,\, \ldots,\,s_{k} ] (x) = \frac{M^{\ast}[s_{j+1},\,\ldots,\,s_{k}](t)}{1+ s_j M^{\ast}[s_{j+1},\,\ldots,\,s_{k}](t)}
\label{basicRecurrenceRelation}
\end{equation}
where
\begin{equation}
x = \int_0^t \left \{ 1+ s_j \,M^{\ast}[s_{j+1},\,\ldots,\,s_{k}](\tau) \right \}^2\,d \tau\,.
\label{xInTermsOft}
\end{equation}
\label{recurrenceTheorem}
\end{theorem}

\begin{proof}
Consider the case where $k-j$ is {\em odd}. Then
$$
W[s_j,\,\ldots,\,s_k] = \frac{s_j}{-z} + \frac{1}{\displaystyle s_{j+1}+\frac{1}{\displaystyle \frac{s_{j+2}}{-z}+ \cdots+ \frac{1}{\displaystyle \frac{s_{k-1}}{-z}+\frac{1}{\displaystyle s_k}}}}\,.
$$ 
An easy calculation shows that
$$
\left ( \widehat{W} \right )^\ast [s_j,\,\ldots,\,s_k] = W[s_{j+1},\, \ldots,\,s_k]\,.
$$
By taking the dual on both sides, we obtain
$$
\widehat{W}[s_j,\,\ldots,\,s_k] = W^\ast [s_{j+1},\, \ldots,\,s_k]\,.
$$
It then follows from Proposition \ref{zeroMassProposition} that
$$
M^\ast [s_{j+1},\,\ldots,\,s_k] (x) = \frac{M[s_j,\,\ldots,\,s_k](t)}{1 - s_j M[s_j,\,\ldots,\,s_k](t)}\,,
$$
where
$$
x = \int_0^t \left \{ 1- s_j M[s_j,\,\ldots,\,s_k](\tau) \right \}^2 d \tau\,.
$$
When we ``turn this around'' and express $M[s_j,\,\ldots,\,s_k]$ in terms of $M^\ast [s_{j+1},\,\ldots,\,s_k]$ and $t$ in terms of $x$, we
obtain the desired result.  The case where $k-j$ is even is analogous.
\end{proof}

\begin{cor}
For every $k>0$ and every $0 \le j <k$, $M[s_j,\,\ldots,\,s_k]$ is a discrete string.
\end{cor}

\begin{proof}
Let $k > 0$ and proceed by induction on $j$, starting with $j=k-1$ and going backwards. 
The claim is obviously true for $j=k-1$.
Suppose that $M[s_{j+1},\,\ldots,\,s_k]$ is a discrete string. By Equation (\ref{xInTermsOft}), $x$ is then an increasing continuous piecewise linear function
of $x$,
and the fact that $M[s_{j},\,\ldots,\,s_k]$ is piecewise constant is an obvious consequence of Equation (\ref{basicRecurrenceRelation}).
\end{proof}

This result permits the efficient recovery of the discrete string $M[s_0,\,\ldots,\,s_n]$ via the sequence
\begin{equation}
M [s_{n}] \rightarrow M [s_{n-1},\,s_n] \rightarrow \cdots \rightarrow M[s_0,\,\ldots,\,s_{n}]\,.
\label{recurrenceRelation}
\end{equation}
Each term in that sequence is a right-continuous piecewise constant function with finitely many jumps, so we need
only keep track of the location, say $x_j$, of these jumps, and of the value, say $y_j$, of the function there.
We always include $x_0=0$ and $y_0$, whether or not $y_0 > 0$.
We have
$$
M [s_n](x) = 
\frac{1}{s_n} \;\;\mbox{for every $x \ge 0$}\,.
$$
So we set
\begin{equation}
x_0^{(0)} = 0 \;\; \mbox{and}\;\; y_0^{(0)} = \frac{1}{s_n}\,.
\label{firstIteration}
\end{equation}

Also
$$
M[s_{n-1},s_n](x) = \cases{
0 & \mbox{if $0 \le x < \frac{1}{s_n}$} \\
\frac{1}{s_{n-1}} & \mbox{if $x \ge \frac{1}{s_n}$}
}
$$
So we set
\begin{equation}
x_0^{(1)} = 0,\; x_1^{(1)} = \frac{1}{s_n}\;\;\mbox{and}\;\; y_0^{(1)} = 0,\;y_1^{(1)} = \frac{1}{s_{n-1}}\,.
\label{secondIteration}
\end{equation}
Then, for $k \ge 1$,
we have the recurrence relations
\begin{equation}
x_0^{(2k)} = 0\,.
\label{firstGeneralStartingValueForX}
\end{equation}
\begin{equation}
x_j^{(2k)} = x_{j-1}^{(2k)} + \left [ 1+s_{n-2k} \,x_j^{(2k-1)} \right ]^2 \left [ y_j^{(2k-1)}-y_{j-1}^{(2k-1)}\right ],\; 1 \le j \le k\,.
\label{firstGeneralFormulaForX}
\end{equation}
\begin{equation}
y_j^{(2k)} = \frac{x_{j+1}^{(2k-1)}}{1+s_{n-2k}\,x_{j+1}^{(2k-1)}}\,,\; 0 \le j < k\,.
\label{firstGeneralFormulaForY}
\end{equation}
\begin{equation}
y_k^{(2k)} = \frac{1}{s_{n-2k}}\,.
\label{firstGeneralFinalValueForY}
\end{equation}

\begin{equation}
x_0^{(2k+1)} = 0,\; \; x_1^{(2k+1)} = y_0^{(2k)}\,.
\label{secondGeneralStartingValueForX}
\end{equation}
\begin{equation}
x_{j+1}^{(2k+1)} = x_{j}^{(2k+1)} + \left [ 1+s_{n-2k-1} \,x_{j}^{(2k)} \right ]^2 \left [ y_{j}^{(2k)}-y_{j-1}^{(2k)}\right ],\; 1 \le j \le k\,.
\label{secondGeneralFormulaForX}
\end{equation}
\begin{equation}
y_0^{(2k+1)} = 0\,.
\label{secondGeneralStartingValueForY}
\end{equation}
\begin{equation}
y_j^{(2k+1)} = \frac{x_{j}^{(2k)}}{1+s_{n-2k-1}\,x_{j}^{(2k)}}\,,\; 1 \le j \le k\,.
\label{secondGeneralFormulaForY}
\end{equation}
\begin{equation}
y_{k+1}^{(2k+1)} = \frac{1}{s_{n-2k-1}}\,.
\label{secondGeneralFinalValueForY}
\end{equation}

The computation terminates when the superscript reaches the value $n$.

\section{Diffusion processes}
\label{diffusionSection}

In this section, we describe in broad terms the probabilistic version of Krein's inverse problem expounded by Knight in \cite{Kn}, and illustrate
our inversion algorithm by means of an example taken from \cite{DY1}. The reader unfamiliar with diffusion processes will find a useful summary of
the theory in \cite{BS}.

Consider a one-dimensional diffusion process $X$ started at the origin,
with values in some interval $I \subseteq {\mathbb R}_+$. Suppose that the origin is instantaneously reflecting and that $X$ is in natural scale. Such a process is
completely determined by its infinitesimal generator, i.e. by the differential operator
\begin{equation}
{\mathscr G} := \frac{1}{M'(x)} \frac{d^2}{d x^2}
\label{infinitesimalGenerator}
\end{equation}
acting  on a suitable set of functions. Here, $d M$ is the speed measure, and the domain of ${\mathscr G}$
consists of twice differentiable functions $f :\,I \rightarrow {\mathbb R}_+$ satisfying the condition $f'(0+) =0$; an additional condition may be
imposed on $f$ at the right boundary, depending on the behaviour of the diffusion there. The process
$$
\lim_{\varepsilon \rightarrow 0+} \frac{1}{\varepsilon M'(\varepsilon)} \int_0^t {\mathbf 1}_{[0,\varepsilon)} \left ( X_{u} \right )\, d u 
$$
is called the {\em local time} (at the origin); it measures the time spent by the diffusion $X$ in the vicinity of the origin, up to time $t$.
Its right-continuous inverse
$$
\tau_t := \inf \left \{ s :\, \lim_{\varepsilon \rightarrow 0+} \frac{1}{\varepsilon M'(\varepsilon)} \int_0^s {\mathbf 1}_{[0,\varepsilon)} \left ( X_{u} \right )\, d u  > t \right \}
$$
is called the {\em inverse local time} (at the origin); it 
is a positive, non-decreasing process that jumps at the times when $X$ begins an excursion away from the origin. The height of the jump is the {\em length} of the excursion, i.e. the time that elapses before $X$ returns to its starting point. It may be shown that the inverse local time is in fact a L\'{e}vy process.
From the theory of such processes, one deduces the existence of a {\em L\'{e}vy exponent}
$\Theta (\lambda)$ defined implicitly by
\begin{equation}
{\mathbb E} \left ( - \lambda \tau_t \right ) = \exp \left [ - t \Theta (\lambda) \right ]\;\;\mbox{for $\lambda > 0$}\,.
\label{levyExponent}
\end{equation}
Furthermore, the L\'{e}vy exponent is necessarily of the form
\begin{equation}
\Theta (\lambda) = a \lambda + b + \int_0^\infty \left ( 1 - \e^{-\lambda y} \right )\, d \nu (y)
\label{levyKhintchineFormula}
\end{equation}
for some numbers $a,\,b \ge 0$ and some {\em L\'{e}vy measure} $d \nu$, i.e. a measure such that
$$
\int_0^\infty \min \{ 1,\, y \} \, d \nu (y) < \infty\,.
$$
Knight shows that, if $X$ is a so-called ``gap diffusion'', then the function $M$ appearing in Equation (\ref{infinitesimalGenerator}) may be identified uniquely with a string, still denoted $M$. The L\'{e}vy exponent $\Theta$ is related to the characteristic function $W$ of the string via
\begin{equation}
\Theta (\lambda) = \frac{1}{W(-\lambda)} = \lambda W^\ast (-\lambda)\,.  
\label{spectralCorrespondence1}
\end{equation}
Furthermore
\begin{equation}
a = c^\ast, \;\; b = s^\ast \;\;\mbox{and}\;\;\nu'(y) = \int_0^\infty \zeta \e^{-\zeta y} d \sigma^\ast(\zeta)\,.
\label{spectralCorrespondence2}
\end{equation}
Krein's inverse problem may thus be rephrased as: ``Find the gap diffusion, given the L\'{e}vy exponent of the inverse local time''. 

To give an example,
let the L\'{e}vy exponent be given by Equation (\ref{levyKhintchineFormula}) with $a = b = 0$ and
\begin{equation}
\nu'(y) := C\, \frac{\e^{- \beta y}}{y^{1+\alpha}}\,, \;\; \alpha \in (0,1)\,, \;\; \beta > 0\,.
\label{donatiLevyMeasure}
\end{equation}
Donati--Martin and Yor showed that the corresponding diffusion is, up to a homeomorphism, a Bessel process with drift \cite{DY1,Wat}. 
The corresponding string may be expressed in terms of the modified Bessel functions.
Set
\begin{equation}
W(z) := \frac{1}{\Theta(-z)} = \frac{\alpha/\gamma}{\left ( 1- z/\beta \right )^\alpha-1}
\label{besselWithDriftCharacteristicFunction}
\end{equation}
where
$$
\gamma  := C\,\Gamma (1-\alpha)\,\beta^\alpha\,.
$$
By using the well-known continued fraction expansion for the binomial (see for instance \cite{Wal}, p. 343), it is easy to see
that $W$ is expressible in the form (\ref{katsContinuedFraction}) with
\begin{equation*}
s_{2j} = \frac{\beta}{\gamma}\, \frac{(1-\alpha)_{j-1}}{(1+\alpha)_{{j-1}}}\, (2 j+1) 
\;\;\mbox{and}\;\;
s_{2j+1} = 2 \gamma \, \frac{(1+\alpha)_{j-1}}{(1-\alpha)_j}
\end{equation*}
where $(\cdot)_n$ denotes Pochhammer's symbol and we use the convention $(\cdot)_n=1$ for $n<0$. Stieltjes' criterion
(\ref{stieltjesSeries}) implies that this continued fraction converges. Hence $W$ belongs to the determinate
subclass of ${\mathbb W}^-$. In what follows, we find approximations of the string $M$ by truncating the continued fraction
after $n$ terms, and computing the corresponding discrete string by the algorithm of the previous section.

\begin{example}
Take
$$
\alpha = \frac{1}{2}, \;\; \beta = 2 \;\;\mbox{and}\;\; C = \frac{1}{\sqrt{2\pi}}\,.
$$
Then
\begin{equation}
M(x) = \frac{2 x}{1+4 x}
\label{brownianMotionWithDrift}
\end{equation}
and the diffusion is a Brownian motion with drift $-2$ in natural scale \cite{BS}. The dots in 
Figure \ref{brownianWithDriftFigure} are the points $(x_j^{(n)},y_j^{(n)})$ such that $x_j^n < 5$ corresponding
to the discrete string $M[s_0,\,\ldots,\,s_n]$ constructed by the algorithm of \S \ref{inverseSection}. 
The superimposed continuous curve is a plot of $M$.
\begin{figure}[htbp]
\vspace{8cm}  
\includegraphics{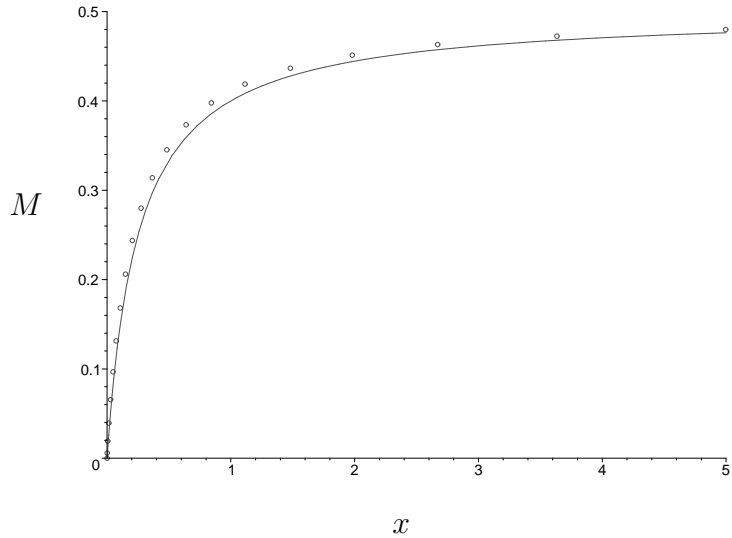} 
\begin{picture}(0,0)
\put(220,0){$x$}
\put(75,120){$M$}
\end{picture}
\caption{The points $(x_n^{(j)},y_n^{(j)})$, shown as dots, corresponding to the approximation $M[s_0,\,\ldots,\,s_n]$ 
of the string $M$ defined by Equation (\ref{brownianMotionWithDrift}). Here, $n=511$ and the solid curve is a plot
of the true $M$.} 
\label{brownianWithDriftFigure} 
\end{figure}
\label{brownianMotionWithDriftExample}
\end{example}

\begin{example}
By taking
$$
C = \frac{1}{2 \Gamma (1-\alpha) \beta^\alpha}
$$
and letting $\alpha \rightarrow 0+$, we obtain
\begin{equation}
W(z) = \frac{2}{\ln \left ( 1- z/\beta \right )}\,.
\label{logarithmicFunction}
\end{equation}
The string with this characteristic function was found explicitly by Donati-Martin and Yor \cite{DY1}: let $Y$ be the ${\mathbb R}_+$-valued diffusion process,
reflected at the origin,
with infinitesimal generator
$$
{\mathscr G}_Y := \frac{1}{2} \frac{d^2}{d y^2} + \left [ \frac{1}{2y} + \sqrt{2 \beta} \frac{K_0 ' \left ( \sqrt{2 \beta} y \right )}{K_0 \left ( \sqrt{2 \beta} y \right )} \right ] \frac{d}{dy}\,.
$$
The scale function of $Y$ is 
$$
S(y) := \int_0^y \frac{dt}{t K_0^2 \left ( \sqrt{2 \beta} t \right )}\,.
$$
The diffusion
$$
X = S(Y)
$$
is in natural scale and its generator is given by Equation (\ref{infinitesimalGenerator}), where $M$ is the string corresponding to
(\ref{logarithmicFunction}). Figure \ref{besselWithDriftFigure} shows a plot of the string for $x$ small, together with an approximation obtained by our algorithm.
\begin{figure}[htbp]
\vspace{8cm}  
\includegraphics{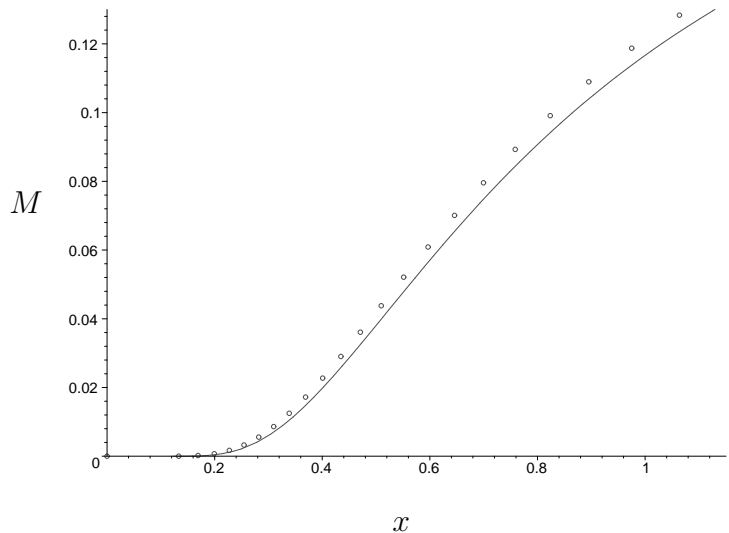} 
\begin{picture}(0,0)
\put(220,0){$x$}
\put(75,120){$M$}
\end{picture}
\caption{The points $(x_n^{(j)},y_n^{(j)})$, shown as dots, corresponding to the approximation $M[s_0,\,\ldots,\,s_n]$ 
of the string $M$ whose characteristic function is given by Equation (\ref{logarithmicFunction}).  $\beta=2$, $n=2047$ and the solid curve is a plot
of the true $M$.} 
\label{besselWithDriftFigure} 
\end{figure}
\label{alphaIsZeroExample}
\end{example}

\section{Concluding remarks}
\label{conclusionSection}

We end with some comments on the numerical issues arising from the algorithm: by design, the approximation is {\em exact} in the limit $x \rightarrow \infty$.
More generally,
our computations suggest that
$$
\max_{j} \left | y_n^{(j)} - M \left ( x_n^{(j)} \right ) \right | = O \left ( n^{-1/2} \right ) \;\;\mbox{as $n \rightarrow \infty$}\,.
$$
There is also some numerical evidence that greater accuracy may be obtained by averaging at the jumps, namely
$$
\max_{j} \left | y_n^{(j-\frac{1}{2})} - M \left ( x_n^{(j)} \right ) \right | = O \left ( n^{-1} \right ) \;\;\mbox{as $n \rightarrow \infty$}
$$
where
$$
y_n^{(j-\frac{1}{2})} := \frac{y_n^{(j-1)}+y_n^{(j)}}{2}\,.
$$

As one would expect,
in order to compute the discrete string from the truncated continued fraction, it is necessary to have good approximations of the coefficients.
It will seldom be the case that exact formulae are available. Instead, one will need
to compute these coefficients by using some numerical algorithm; see for instance \cite{CT} and the comments therein.

Finally, the inversion method we have described generates only piecewise constant approximations of $M$, and so the recovery of $M'$--- 
in cases where the string is absolutely continuous---
is not entirely straightforward. It would be of interest to adapt to our case the ingenious recovery technique
that Borcea \etal \cite{BDK} devised in the context of the Sturm--Liouville problem in impedance form.

\ack It is a pleasure to thank Alain Comtet for his encouragement and for discussions of the material, and the Laboratoire de Physique Th\'{e}orique et
Mod\`{e}les Statistiques, Universit\'{e} Paris-Sud, for its hospitality while some of this work was carried out. Thanks also to the anonymous referees,
whose thoughtful comments and suggestions helped improve an earlier version of the paper.

\end{document}